\documentclass[a4paper,11pt]{article}

\usepackage[utf8]{inputenc}
\usepackage[T1]{fontenc}
\usepackage{lmodern}
\usepackage{microtype}
\usepackage[a4paper,textwidth=16cm,textheight=22cm,centering]{geometry}

\usepackage{mathtools}
\usepackage{amssymb,amsfonts}
\usepackage{amsthm}
\usepackage{setspace}
\usepackage{indentfirst}
\usepackage{enumitem}
\usepackage{float}
\usepackage{needspace}

\usepackage{cite}
\usepackage{color}
\usepackage{xcolor}
\usepackage{aliascnt}
\usepackage{hyperref}
\hypersetup{
  colorlinks=true,
  linkcolor=red,
  citecolor=blue,
  urlcolor=blue,
  pdftitle={Sharp connectivity thresholds for mixed rigidity packings and improved bounds for highly connected orientations},
  pdfauthor={Hanzhi Bai, Jorgen Bang-Jensen, and Jin Yan}}

\theoremstyle{plain}
\newtheorem{theorem}{Theorem}[section]
\newaliascnt{lemma}{theorem}
\newtheorem{lemma}[lemma]{Lemma}
\aliascntresetthe{lemma}
\newaliascnt{claim}{theorem}
\newtheorem{claim}[claim]{Claim}
\aliascntresetthe{claim}
\newaliascnt{fact}{theorem}

\aliascntresetthe{fact}
\newaliascnt{corollary}{theorem}
\newtheorem{corollary}[corollary]{Corollary}
\aliascntresetthe{corollary}
\newaliascnt{proposition}{theorem}
\newtheorem{proposition}[proposition]{Proposition}
\aliascntresetthe{proposition}
\newaliascnt{conjecture}{theorem}
\newtheorem{conjecture}[conjecture]{Conjecture}
\aliascntresetthe{conjecture}
\newaliascnt{observation}{theorem}

\aliascntresetthe{observation}

\theoremstyle{definition}
\newaliascnt{definition}{theorem}

\aliascntresetthe{definition}
\newaliascnt{problem}{theorem}
\newtheorem{problem}[problem]{Problem}
\aliascntresetthe{problem}
\newaliascnt{example}{theorem}

\aliascntresetthe{example}
\newaliascnt{construction}{theorem}

\aliascntresetthe{construction}

\theoremstyle{definition}
\newaliascnt{remark}{theorem}

\aliascntresetthe{remark}
\theoremstyle{remark}
\newaliascnt{question}{theorem}

\aliascntresetthe{question}
\usepackage{thmtools}
\newcommand{\Rmat}{\mathcal{R}}
\newcommand{\dotunion}{\mathbin{\dot\cup}}
\newcommand{\jbj}[1]{\textcolor{black}{#1}}

\newcommand{\bhz}[2]{\textcolor{black}{#1}}

\begin{document}

\date{}
\begin{spacing}{1.03}

\title{Sharp connectivity thresholds for mixed rigidity packings and \jbj{improved} bounds for highly connected orientations \jbj{of graphs}}

\author{
Hanzhi Bai$^{1}$\and
J{\o}rgen Bang-Jensen$^{2}$\and
Jin Yan$^{1}$
}

\footnotetext[1]{{School of Mathematics, Shandong University, Jinan 250100,
P.~R.~China. E-mails: {\tt bhz@mail.sdu.edu.cn} (H. Bai),
{\tt yanj@sdu.edu.cn} (J. Yan). Supported by the National Natural Science
Foundation of China (Grant No.~12571373) and the Natural Science Foundation
of Shandong Province (Grant No.~ZR2025MS05).}}
\footnotetext[2]{Department of Mathematics and Computer Science, University of
Southern Denmark, Odense DK-5230, Denmark \jbj{and School of Mathematics,
Shandong University, Jinan 250100, P.~R.~China}. E-mail:
{\tt jbj@imada.sdu.dk} (J. Bang-Jensen).
}

\maketitle

\begin{abstract}
Garamv\"olgyi, Jord\'an, Kir\'aly and Vill\'anyi
[{{\bf Forum Math. Pi} \textbf{13} (2025), Paper No.~e11}]
posed two sharp connectivity conjectures for packing rigid spanning
subgraphs: one for the equal-dimensional case and the other for the packing
of a $d$-rigid spanning subgraph with a spanning tree. We prove a unified
theorem: for arbitrary positive integers $d_1,\ldots,d_s$, every
$\sum_{i=1}^{s}d_i(d_i+1)$-connected graph contains pairwise edge-disjoint
spanning subgraphs $H_1,\ldots,H_s$ such that $H_i$ is $d_i$-rigid for every
$i$. The connectivity bound is sharp whenever
$\sum_{i=1}^{s}d_i(d_i+1)\ge4$.

As special cases, the theorem settles both conjectures, confirms the
conjecture of Garamv\"olgyi, Jord\'an and Kir\'aly
[{\bf J. Combin. Theory Ser. B} \textbf{166} (2024), 1--29] that every
$tk(k+1)$-connected graph contains $t$ pairwise edge-disjoint
$k$-connected spanning subgraphs, and gives the sharp threshold
$d(d+1)+2r$ for packing one $d$-rigid spanning subgraph together with $r$
pairwise edge-disjoint spanning trees.

We also obtain two upper bounds related to Thomassen's conjecture on highly
connected orientations \jbj{of graphs}. If $f(q)$ is the least integer such
that every $f(q)$-connected graph has a $q$-connected orientation, then
$f(q)\le(25q^2+41q-16)/2$ for every $q\ge3$ and
$f(q)\le8q^2+212q+1404=(8+o(1))q^2$ for all sufficiently large $q$;
\jbj{these two results} reduce the leading coefficient in the previous
quadratic bound from $320$ to
\jbj{$25/2$ for every $q\ge3$ and $8$ for all sufficiently large $q$. Compared to the bound for $f(q)$ obtained by Garamv\"olgyi et al. we obtain the better bounds, not only through our tight rigidity result but also by exploiting the leftover edges when we remove two edge-disjoint spanning (sufficiently) rigid graphs.}
\end{abstract}

\jbj{highly connected orientations}.


\section{Introduction}

Generic $d$-dimensional rigidity may be formulated through the generic
rigidity matroid: a graph is {\bf $d$-rigid} if its edge set spans this
matroid; see \cite{AsimowRoth,Graver,GraverServatius,SchulzeWhiteley}. For
$d=1$, rigidity is ordinary connectivity. Laman \cite{Laman} characterized
minimal rigidity in dimension two, while no analogous combinatorial
characterization is known for $d\ge3$ \cite{SchulzeWhiteley,GJKV}.

Lov\'asz and Yemini \cite{LovaszYemini} proved the sharp
$6$-connectivity threshold for $d=2$. Vill\'anyi \cite{Villanyi} later proved
the general sharp threshold $d(d+1)$ for $d\ge2$, with the matching lower
construction given by Garamv\"olgyi, Jord\'an, Kir\'aly and Vill\'anyi
\cite{GJKV}.

Packing several rigid spanning subgraphs is a rigidity counterpart of the
classical spanning-tree packing problem of Nash-Williams and Tutte
\cite{NashWilliams,Tutte}. Jord\'an \cite{Jordan} proved the sharp
$6t$-connectivity threshold for packing $t$ edge-disjoint $2$-rigid spanning
subgraphs.

A related result using the $C^1_2$-cofactor matroid, a rigidity-type matroid,
shows that every
$12t$-connected graph contains $t$ edge-disjoint $C^1_2$-rigid, and hence
$3$-connected, spanning subgraphs \cite{GJKCofactor}. In arbitrary dimension,
Garamv\"olgyi, Jord\'an, Kir\'aly and Vill\'anyi \cite{GJKV} proved that
$10td(d+1)$-connectivity is sufficient. Their construction shows that
$td(d+1)-1$ is insufficient and led to the following conjecture.

\begin{conjecture}
\label{conj:sharp}
For all integers $d\ge2$ and $t\ge1$, every $td(d+1)$-connected graph
contains $t$ pairwise edge-disjoint $d$-rigid spanning subgraphs.
\end{conjecture}

For $d=2$, Cheriyan, Durand de Gevigney and Szigeti
\cite{CheriyanDurandSzigeti} proved that $(6+2r)$-connectivity suffices to
pack, pairwise edge-disjointly, one $2$-rigid spanning subgraph and $r$
spanning trees. For general $d$,
Garamv\"olgyi, Jord\'an, Kir\'aly and Vill\'anyi
\cite[Theorem~1.8]{GJKV} proved that $(d^2+3d+5)$-connectivity suffices to
pack, edge-disjointly, one $d$-rigid spanning subgraph and one spanning tree,
and proposed the following sharp threshold.

\begin{conjecture}
\label{conj:mixed}
For every integer $d\ge2$, every $(d(d+1)+2)$-connected graph contains
edge-disjoint spanning subgraphs $H$ and $T$ such that $H$ is $d$-rigid and
$T$ is a spanning tree.
\end{conjecture}

Except for a single one-dimensional factor, the lower-bound construction in
\cite[Lemma~6.1]{GJKV} already applies to arbitrary prescribed dimensions and
suggests the following unified packing theorem, which settles
Conjectures~\ref{conj:sharp} and~\ref{conj:mixed}.

\begin{restatable}{theorem}{mainthm}\label{thm:main}
Let $d_1,\ldots,d_s$ be positive integers and put
\[
K=\sum_{i=1}^{s}d_i(d_i+1).
\]
Every $K$-connected graph \jbj{$G$} contains pairwise edge-disjoint spanning subgraphs
$H_1,\ldots,H_s$ such that $H_i$ is $d_i$-rigid for every
$i\in\{1,\ldots,s\}$. If $K\ge4$, this connectivity bound is best possible.
\end{restatable}

The equal-dimensional specialization of Theorem~\ref{thm:main} has
consequences for both rigid and connected spanning subgraphs. For every
positive integer $k$, let $f^{*}(k)$ be the least integer such that, for every
positive integer $t$, every $t f^{*}(k)$-connected graph contains $t$
pairwise edge-disjoint $k$-connected spanning subgraphs. Garamv\"olgyi,
Jord\'an and Kir\'aly \cite[Section~5.3]{GJKCofactor} conjectured that
$f^{*}(k)\le k(k+1)$, while Garamv\"olgyi, Jord\'an, Kir\'aly and Vill\'anyi
\cite{GJKV} proved $f^{*}(k)\le10k(k+1)$.

\smallskip

\jbj{Theorem \ref{thm:main} has several consequences.}
\begin{corollary}\label{cor:equal-packings}
For every positive integer $t$, the following statements hold.
\begin{enumerate}[label=\textup{(\alph*)}]
\item For every integer $d\ge2$, every $td(d+1)$-connected graph contains
$t$ pairwise edge-disjoint $d$-rigid spanning subgraphs. The connectivity
bound is best possible.
\item For every positive integer $k$, every $tk(k+1)$-connected graph
contains $t$ pairwise edge-disjoint $k$-connected spanning subgraphs.
Consequently, $f^{*}(k)\le k(k+1)$.
\end{enumerate}
\end{corollary}

Part~\textup{(a)} of Corollary~\ref{cor:equal-packings} confirms
Conjecture~\ref{conj:sharp}. Part~\textup{(b)} confirms the bound
$f^{*}(k)\le k(k+1)$ proposed by Garamv\"olgyi, Jord\'an and Kir\'aly
\cite[Section~5.3]{GJKCofactor} and improves the previous bound
$10k(k+1)$ to $k(k+1)$ for $k\ge4$.

The specialization \jbj{$s=r+1$ and $d_1=d,d_2=1,\ldots,d_{s}=1$} yields the following mixed packing \jbj{result}, which, \jbj{through the fact that every $d$-rigid graph on at least $d+1$ vertices is $d$-connected,} this also gives bounds for Kriesell's removable-spanning-tree problem
\cite[Problem~444]{MoharNowakowskiWest}.

\begin{corollary}\label{cor:tree-removal}
For every positive integer $r$, the following statements hold.
\begin{enumerate}[label=\textup{(\alph*)}]
\item For every integer $d\ge2$, every $(d(d+1)+2r)$-connected graph
contains pairwise edge-disjoint spanning subgraphs
$H,T_1,\ldots,T_r$, where $H$ is $d$-rigid and
$T_1,\ldots,T_r$ are spanning trees. The connectivity bound is best
possible.
\item For every positive integer $k$, every $(k(k+1)+2r)$-connected graph
$G$ contains $r$ pairwise edge-disjoint spanning trees
$T_1,\ldots,T_r$ such that
\(
G-\bigcup_{i=1}^{r}E(T_i)
\)
is $k$-connected.
\end{enumerate}
\end{corollary}

Part~\textup{(a)} of Corollary~\ref{cor:tree-removal}, with $r=1$,
settles Conjecture~\ref{conj:mixed}, and the full statement is sharp for
every $r$. Part~\textup{(b)} improves the previous removable-tree bound for
$r=1$ and extends it to the simultaneous deletion of $r$ trees.

The proof \jbj{exploits} the following mixed-dimensional rank inequality. For
$D_i=\binom{d_i+1}{2}$ and an $n$-vertex graph $H$ \jbj{on $n\geq d_i+1$ vertices}, the quantity
$d_i n-D_i-r_{d_i}(H)$ is the corresponding rank deficiency \jbj{of $H$.}

\begin{restatable}{theorem}{rankdefectthm}\label{thm:rank-defect}
Let $d_1,\ldots,d_s$ be positive integers, put
$K=\sum_{i=1}^{s}d_i(d_i+1)$ and $D_i=\binom{d_i+1}{2}$, and let
$H=(V,E_H)$ and $F=(V,E_F)$ be edge-disjoint graphs on the same $n$-vertex
set. If $\jbj{G=}H\cup F$ is $K$-connected, then
\begin{equation}\label{eq:rank-defect}
|E_F|+\sum_{i=1}^{s}r_{d_i}(H)
\ge\sum_{i=1}^{s}(d_i n-D_i).
\end{equation}
\end{restatable}

By \jbj{Edmonds'} matroid union theorem, Theorem~\ref{thm:rank-defect}, applied to
every edge partition \jbj{of $E(G)$}, yields the packing in Theorem~\ref{thm:main}. The full
argument is given in Subsection~\ref{subsec:proof-rank-packing}.

Our further results concern highly connected orientations. An orientation of
a graph is {\bf $q$-connected} if the resulting digraph has at least $q+1$
vertices and remains strongly connected after the deletion of any set of
fewer than $q$ vertices. Thomassen \cite[Conjecture~10]{Thomassen1989} posed
the following conjecture.

\begin{conjecture}
\label{conj:thomassen}
For every positive integer $q$, there exists a least integer $f(q)$ such that
every $f(q)$-connected graph has a $q$-connected orientation.
\end{conjecture}

\jbj{It follows by the well-known strong orientation theorem by} Robbins \cite{Robbins} that $f(1)=2$. Thomassen and Jackson conjectured
that $f(q)=2q$ for every positive integer $q$
\cite[Conjecture~11]{Thomassen1989}. \jbj{Jord\'an \cite{Jordan} confirmed that $f(2)$ is finite by proving that $f(2)\leq 18$. Later } Thomassen
\cite{Thomassen2015} proved that $f(2)=4$. \jbj{Finally, recently } Garamv\"olgyi, Jord\'an, Kir\'aly and
Vill\'anyi \cite{GJKV} proved that $f(q)$ exists \jbj{for all positive integers $q$ by showing that} 
$f(q)\le20(4q-4)(4q-3)=320q^2-560q+240$ for $q\ge2$. We obtain the following
two improvements.

\begin{restatable}{theorem}{explicitorientationthm}\label{thm:orientation-explicit}
For every integer $q\ge3$,
\[
f(q)\le\frac{25q^2+41q-16}{2}.
\]
\end{restatable}

\begin{restatable}{theorem}{orientationthm}\label{thm:orientation}
For all sufficiently large integers $q$,
\[
f(q)\le 8q^2+212q+1404.
\]
In particular, $f(q)\le(8+o(1))q^2$.
\end{restatable}

Thus the previous leading coefficient $320$ is reduced to $25/2$ for every
$q\ge3$ and to $8$ asymptotically for all sufficiently large $q$.

The stronger Thomassen--Jackson conjecture $f(q)=2q$ \jbj{\cite{Thomassen1989}  remains open for all $q>2$.}

\paragraph{Organization of the paper.}
Section~\ref{sec:preliminaries} presents the notation and preliminaries, and
Section~\ref{sec:packing} proves Theorems~\ref{thm:rank-defect}
and~\ref{thm:main} and some corollaries. Sections~\ref{sec:explicit-orientation}
and~\ref{sec:orientation} prove Theorems~\ref{thm:orientation-explicit}
and~\ref{thm:orientation}, respectively, and Section~\ref{sec:remarks} gives
concluding remarks and further consequences.

\section{Preliminaries}\label{sec:preliminaries}

\subsection{Notation}

All graphs in this paper are finite and simple. For a graph $G$, its vertex
set and edge set are denoted by $V(G)$ and $E(G)$, \jbj{respectively}, and
its order is $|V(G)|$. For $S\subseteq V(G)$, let $G[S]$ be the subgraph
induced by $S$ and put $G-S=G[V(G)\setminus S]$; we abbreviate
$G-\{v\}$ to $G-v$. For $F\subseteq E(G)$, let
$G-F=(V(G),E(G)\setminus F)$, and write $G-e$ when $F=\{e\}$.
We write $K(V)$ for the complete graph on vertex set $V$ and $K_n$ for the
complete graph of order $n$. For $X\subseteq V(G)$, let
$N_G(X)=\{v\in V(G)\setminus X:uv\in E(G)\text{ for some }u\in X\}$, put
$N_G(v)=N_G(\{v\})$, and write $d_G(v)=|N_G(v)|$ and
$\delta(G)=\min_{v\in V(G)}d_G(v)$. For disjoint sets
$A,B\subseteq V(G)$, let $E_G(A,B)$ be the set of edges with one end in each
set. The notation $A\dotunion B$ denotes a disjoint union. If $H$ and $F$
have the same vertex set, then $H\cup F$ has edge set
$E(H)\cup E(F)$. For a positive integer $k$, a graph $G$ is
{\bf $k$-connected} if $|V(G)|\ge k+1$ and $G-S$ is connected whenever
$S\subseteq V(G)$ and $|S|<k$.

For a matroid $M$ on a finite ground set $E$, its rank $r_M(X)$ is the
maximum size of an independent subset of $X\subseteq E$, and a {\bf base}
\jbj{of $M$} is an independent set of size $r_M(E)$.

A $d$-dimensional realization of a graph $G$ is a map
$p:V(G)\to\mathbb{R}^d$; it is {\bf generic} if its coordinates are
algebraically independent over $\mathbb{Q}$. The rigidity matrix $R(G,p)$
has one row for each edge, and at a generic realization its row-dependence
relations depend only on the graph. At a generic realization of $K(V)$, the
row matroid of $R(K(V),p)$ is the
{\bf $d$-dimensional generic rigidity matroid}
$\Rmat_d(K(V))$, with rank function $r_d$; see
\cite{AsimowRoth,Graver,SchulzeWhiteley}. For
\jbj{$H=(V,E_H)$}, let $\Rmat_d(H)$ be the
\jbj{{\bf restriction}} of $\Rmat_d(K(V))$ to $E_H$; its independent sets
are precisely the subsets of $E_H$ independent in $\Rmat_d(K(V))$. We write
$r_d(H)=r_d(E_H)$. For $U\subseteq V$, the notation
$\Rmat_d(K(U))$ has the analogous meaning. When $|V|\ge d$,
\begin{equation}\label{eq:complete-rank}
r_d(K(V))=d|V|-\binom{d+1}{2}.
\end{equation}
The graph $H$ is {\bf $d$-rigid} if $r_d(H)=r_d(K(V))$, and it is
{\bf minimally $d$-rigid} if $H$ is $d$-rigid but $H-e$ is not $d$-rigid
for every $e\in E_H$. Thus the edge set of a minimally $d$-rigid graph is a
base of $\Rmat_d(K(V))$. For $d=1$, $\Rmat_1(G)$ is the graphic matroid, so
$1$-rigidity is ordinary connectivity.

We next state the results used in the proofs.

\subsection{Rigidity and \jbj{matroids}}

We use the standard fact that generic rigidity implies the corresponding
vertex-connectivity; see, for example, Graver, Servatius and Servatius
\cite{GraverServatius} or Whiteley \cite{Whiteley}.

\begin{lemma}\label{lem:standard}
Let $d\ge2$. Every $d$-rigid graph on at least $d+1$ vertices is
$d$-connected.
\end{lemma}

\jbj{Conversely, sufficiently high vertex-connectivity guarantees rigidity,
with a sharp threshold that is quadratic in the dimension.} Vill\'anyi
\cite[Theorem~1.1]{Villanyi} proved the following result.

\begin{theorem}\label{thm:villanyi-rigidity}
For every integer $d\ge2$, every $d(d+1)$-connected graph is $d$-rigid.
\end{theorem}

We use the following rank form of the matroid union theorem of Edmonds
\cite{Edmonds,EdmondsPartition}; see also Oxley \cite{Oxley}.

\begin{theorem}\label{thm:edmonds}
Let $M_1,\ldots,M_t$ be matroids on a common finite ground set $E$, with rank
functions $r_1,\ldots,r_t$. The maximum size of a set that can be partitioned
as $I_1\dotunion\cdots\dotunion I_t$, where $I_i$ is independent in $M_i$,
is
\[
\min_{X\subseteq E}\left(|E\setminus X|+\sum_{i=1}^{t}r_i(X)\right).
\]
\end{theorem}

\subsection{Packings and orientations}

Garamv\"olgyi, Jord\'an, Kir\'aly and Vill\'anyi \cite[Lemma~6.1]{GJKV}
established the following general lower-bound construction.

\begin{lemma}\label{lem:lower-bound-construction}
Let $s\ge1$ and let $d_1,\ldots,d_s$ be positive integers. Set
$K=\sum_{i=1}^{s}d_i(d_i+1)-1$, and assume $K\ge3$. There exist
infinitely many $K$-connected
graphs $G$ that do not contain pairwise edge-disjoint spanning subgraphs
$G_1,\ldots,G_s$ such that $G_i$ is $d_i$-rigid for every
$i\in\{1,\ldots,s\}$.
\end{lemma}

Nash-Williams and Tutte \cite{NashWilliams,Tutte} proved the spanning-tree
packing theorem; we use the following form.

\begin{theorem}\label{thm:tree-packing}
Let $G$ be a graph and let $t\ge1$. If $G$ is $2t$-edge-connected, then it
contains $t$ pairwise edge-disjoint spanning trees. In particular, the same
conclusion holds when $G$ is $2t$-connected.
\end{theorem}

Hakimi \cite{Hakimi} proved the following criterion for
\jbj{the existence of orientations with} prescribed in-degrees.

\begin{theorem}\label{thm:hakimi}
Let $F=(V,E)$ be a graph and let $g:V\to\mathbb Z_{\ge0}$. There is an
orientation of $F$ with in-degree $g(v)$ at every $v\in V$ if and only if
$|E(F[X])|\le g(X)$ for every $X\subseteq V$ and $|E|=g(V)$, where
$g(X)=\sum_{v\in X}g(v)$. The same criterion holds for prescribed
out-degrees.
\end{theorem}

We also use two standard results on graph linkages. A {\bf linkage problem}
of order $r$ is a collection $\{\{s_i,t_i\}:1\le i\le r\}$ of pairs of
vertices. A solution consists of paths joining $s_i$ to $t_i$ that are
pairwise internally disjoint and may meet only at common prescribed ends.
A graph is {\bf strongly $r$-linked} if every linkage problem of order $r$
has a solution; it is {\bf $r$-linked} if this is required when the $2r$
prescribed ends are distinct.

Kawarabayashi and Mohar \cite[Theorem~5.2]{KawarabayashiMohar} proved the
first part of the next lemma, and the final assertion follows from the
linkage theorem of Thomas and Wollan \cite{ThomasWollan}.

\begin{theorem}\label{thm:linked-subgraph}
Let $r\ge1$ and let $F$ be an $n$-vertex graph. If
$n\ge5r/2$ and
\[
|E(F)|\ge \frac{25}{4}rn-\frac{25}{2}r^2,
\]
then $F$ contains a $2r$-connected subgraph $L$ satisfying
$|E(L)|\ge5r|V(L)|$. In particular, $L$ is $r$-linked.
\end{theorem}

Mader \cite{MaderStrong} proved that ordinary linkedness already gives the
version with repeated prescribed ends; see also Kawarabayashi, Lee, Reed and
Wollan \cite{KawarabayashiLeeReedWollan}.

\begin{theorem}\label{thm:strongly-linked}
Every $r$-linked graph on at least $2r$ vertices is strongly $r$-linked.
\end{theorem}

\subsection{Vertex connectivity}

We use the terminology and notation of Jackson and Jord\'an
\cite{JacksonJordanAug}. Let $s\ge2$ be an integer, and let $J=(W,E)$ be a graph with
$|W|\ge s+1$. For a nonempty set $X\subsetneq W$, put
$X^{*}=W\setminus(X\cup N_J(X))$.
The set $X$ is a {\bf fragment} of $J$ if $X^{*}\ne\varnothing$, and it is
{\bf $s$-deficient} if $|N_J(X)|<s$. A fragment $X$ {\bf separates} an
unordered pair $\{u,v\}$ if one of $u,v$ belongs to $X$ and the other belongs
to $X^{*}$. A family of fragments is {\bf half-disjoint} if every unordered
pair of vertices is separated by at most two members of the family.
Thus $N_J(X)$ separates $X$ from the nonempty set $X^{*}$, and
$s-|N_J(X)|$ measures how far this separator falls below the target size
$s$. Half-disjointness ensures that the endpoints of each added edge are
separated by at most two fragments in the family.

Let
\begin{equation}\label{eq:t-prime-definition}
t'_s(J)=\max_{\mathcal A}
\sum_{X\in\mathcal A}\bigl(s-|N_J(X)|\bigr),
\end{equation}
where $\mathcal A$ ranges over half-disjoint families of $s$-deficient
fragments. Similarly, let
\begin{equation}\label{eq:t-definition}
t_s(J)=\max_{\mathcal X}
\sum_{X\in\mathcal X}\bigl(s-|N_J(X)|\bigr),
\end{equation}
where $\mathcal X$ ranges over pairwise disjoint families of $s$-deficient
fragments. The empty family is allowed in both definitions. Since every
pairwise disjoint family of fragments is half-disjoint, $t_s(J)\le t'_s(J)$.

Let $a_s(J)$ be the minimum number of edges that must be added to $J$
to make it $s$-connected. Thus $a_s(J)$ is the minimum size of a set
$A\subseteq E(K(W))\setminus E(J)$ for which the graph $J+A$, with edge set
$E(J)\cup A$, is $s$-connected. If $K\subseteq W$ has order $s-1$ and $J-K$ has at
least three components, then $K$ is an {\bf $(s-1)$-shredder}. Let
\jbj{$b_J(K)$ denote} the number of components of $J-K$, and put
$\delta_J(K)=\max\{0,\max_{x\in K}(s-d_J(x))\}$. Define
$\widehat b_s(J)=\max_K(b_J(K)+\delta_J(K))$, where the maximum is over all
$(s-1)$-shredders; if there is no such shredder, set
$\widehat b_s(J)=0$.

Jackson and Jord\'an \cite[Theorem~7.7]{JacksonJordanAug} proved a min--max
theorem on the minimum number of edges needed to make a graph
$s$-connected. \jbj{We apply their theorem to disconnected graphs in
Section~\ref{sec:packing}.} In their notation, take the target connectivity
$k=s$ and the initial connectivity $\ell=0$. Since
$s\ge2$, the condition $\ell\le k-2$ holds. With the standard convention
that every nonempty graph is $0$-connected, this gives the following form
for disconnected graphs.

\begin{theorem}\label{thm:JJ-connectivity}
Let $s\ge2$ be an integer and put $C_s=10(s+2)^3(s+1)^3$. Every
disconnected graph $J$ of order at least $s+1$ with $a_s(J)\ge C_s$ satisfies
\begin{equation}\label{eq:JJ-connectivity}
a_s(J)=
\max\left\{
\widehat b_s(J)-1,
\left\lceil\frac{t_s(J)}{2}\right\rceil
\right\}.
\end{equation}
\end{theorem}

\section{\jbj{Optimal} mixed rigidity packings}\label{sec:packing}

In this section, we prove Theorems~\ref{thm:rank-defect}
and~\ref{thm:main} and some corollaries. We use information about small vertex separators to
bound the number of missing independent rigidity constraints. We first show
that if the edges of $F$ are added to $H$ so that $H\cup F$ is
$s$-connected, then deficient fragments of $H$ force many edges of $F$
across the corresponding separations. To connect this separator count with
rigidity, we take many disjoint labeled copies of the same graph. Any
\jbj{set of}
edges whose addition makes their union $d(d+1)$-connected also makes the
resulting graph $d$-rigid by Theorem~\ref{thm:villanyi-rigidity}.
\bhz{This gives a lower bound, in terms of rigidity rank, on the minimum
number of edges needed to obtain a $d(d+1)$-connected graph.
Theorem~\ref{thm:JJ-connectivity} provides an upper bound for the same
minimum in terms of deficient fragments.}{}

The rank-deficiency estimate is first proved for a single prescribed
dimension $d\ge2$; the case $d=1$ has a direct graph-theoretic proof. A
scaling inequality for deficient fragments then allows us to combine the
estimates for arbitrary dimensions $d_1,\ldots,d_s$. Their coefficients add
to one, which yields Theorem~\ref{thm:rank-defect}. We finally apply
Edmonds' matroid union theorem once to obtain Theorem~\ref{thm:main}.

\subsection{Edge bounds}\label{subsec:edge-bound}

The following lemma uses the counting argument of Jackson and Jord\'an
\cite[p.~36]{JacksonJordanAug} for a lower bound on the number of edges
needed to make a graph $s$-connected.

\begin{lemma}\label{lem:fragment-edge-bound}
Let $s\ge2$ be an integer, and let $H=(V,E_H)$ and $F=(V,E_F)$ be
edge-disjoint graphs on the same vertex set. If $H\cup F$ is $s$-connected,
then
\[
t'_s(H)\le 2|E_F|.
\]
\end{lemma}

\begin{proof}
Let $X$ be an $s$-deficient fragment of $H$, and choose $x\in X$ and
$y\in X^{*}$. The vertices $x$ and $y$ are distinct. Since $H\cup F$ is
$s$-connected, Menger's theorem gives $s$ internally vertex-disjoint
$x$--$y$ paths in $H\cup F$. If $x$ and $y$ are adjacent, the edge $xy$
itself is counted as one of these paths.

Neither endpoint belongs to $N_H(X)$. Since the paths are internally
vertex-disjoint, at most $|N_H(X)|$ of them contain a vertex of $N_H(X)$.
Consider one of the remaining paths, and traverse it from $x$ to $y$. Let
$uv$ be the first edge on this path with $u\in X$ and $v\notin X$. The path
avoids $N_H(X)$, so $v\in X^{*}$. By the definition of $X^{*}$, no edge of
$H$ joins $X$ and $X^{*}$; hence $uv\in E_F(X,X^{*})$. The edges obtained
from distinct paths are distinct. It follows that
\begin{equation}\label{eq:fragment-crossing}
|E_F(X,X^{*})|\ge s-|N_H(X)|.
\end{equation}

Let $\mathcal A$ be a half-disjoint family of $s$-deficient fragments of
$H$. Sum \eqref{eq:fragment-crossing} over $X\in\mathcal A$. An edge
$uv\in E_F$ is included in $E_F(X,X^{*})$ precisely when $X$ separates
$\{u,v\}$. By half-disjointness, this occurs for at most two members of
$\mathcal A$. Therefore
\[
\sum_{X\in\mathcal A}\bigl(s-|N_H(X)|\bigr)\le2|E_F|.
\]
Taking the maximum over all such families $\mathcal A$ proves the lemma.
\end{proof}

\begin{lemma}\label{lem:t-prime-scaling}
Let $b\ge a\ge2$ be integers, and let $H$ have order at least $b+1$. Then
\[
t'_b(H)\ge\frac{b}{a}t'_a(H).
\]
\end{lemma}

\begin{proof}
Let $\mathcal A$ be a half-disjoint family of $a$-deficient fragments of
$H$. Every member of $\mathcal A$ is also $b$-deficient, and the family
remains half-disjoint. For each $X\in\mathcal A$,
\[
b-|N_H(X)|-\frac ba\bigl(a-|N_H(X)|\bigr)
=\frac{b-a}{a}|N_H(X)|\ge0.
\]
Thus the total $b$-deficiency of $\mathcal A$ is at least $b/a$ times its
total $a$-deficiency. Taking a family that attains $t'_a(H)$ proves the
lemma.
\end{proof}

\subsection{Rank bounds}\label{subsec:rank-bound}

For an integer $d\ge2$, put $D=\binom{d+1}{2}$ and
$k_0=d(d+1)=2D$. If $H$ is a graph on $n\ge k_0+1$ vertices,
\jbj{let} $\Delta_d(H)=dn-D-r_d(H)$. Thus $\Delta_d(H)$ is the
number of independent rigidity constraints by which $H$ falls short of full
rank. In particular, it is zero exactly when $H$ is $d$-rigid, and adding
one edge can decrease it by at most one because one edge contributes only one
row to the rigidity matrix.

\begin{lemma}\label{lem:rank-bound}
Let $d\ge2$ be an integer, let $k_0=d(d+1)$, and let $H$ be a graph on
$n\ge k_0+1$ vertices. Then
\begin{equation}\label{eq:rank-bound}
t'_{k_0}(H)\ge2\Delta_d(H).
\end{equation}
\end{lemma}

\begin{proof}
Fix $H$, and write $\tau=t'_{k_0}(H)$ and $\Delta=\Delta_d(H)$. For an
integer $L\ge2$, let $J_L$ be the disjoint union of $L$ labeled copies of
$H$, with vertex sets $V_1,\ldots,V_L$.

We first obtain a lower bound for $a_{k_0}(J_L)$. Choose a set $A$ of
$a_{k_0}(J_L)$ new edges such that $J_L+A$ is $k_0$-connected. Then
$J_L+A$ is $d$-rigid by Theorem~\ref{thm:villanyi-rigidity}. Since $J_L+A$
is $d$-rigid and has $Ln$ vertices, its rigidity rank is
$r_d(J_L+A)=dLn-D$.

We next compute the rank before the edges of $A$ are added. Order the columns
of the rigidity matrix according to $V_1,\ldots,V_L$. Since no edge joins
distinct copies, the matrix is block diagonal, with one block for each copy
of $H$. Equivalently, the constraints from different copies involve disjoint
sets of vertex coordinates. Each block has rank $r_d(H)$, and hence
$r_d(J_L)=Lr_d(H)$. Each edge of $A$ adds one row to the rigidity matrix and
can increase its rank by at most one. Therefore
$dLn-D=r_d(J_L+A)\le Lr_d(H)+|A|$, which gives
\begin{equation}\label{eq:a-lower}
a_{k_0}(J_L)\ge dLn-D-Lr_d(H)=L(D+\Delta)-D.
\end{equation}
Since $r_d(H)\le dn-D$, we have $\Delta\ge0$.  As $D>0$,
\eqref{eq:a-lower} shows that $a_{k_0}(J_L)$ tends to infinity with $L$. \jbj{In particular, we can choose $L$ so that $a_{k_0}(J_L)$ is at least $C_{k_0}=10(k_0+2)^3(k_0+1)^3$ (the value from Theorem \ref{thm:JJ-connectivity}).} 

We next estimate the term involving $t_{k_0}(J_L)$ in
Theorem~\ref{thm:JJ-connectivity}.

\begin{claim}\label{clm:t-bound}
For every integer $L\ge2$,
\begin{equation}\label{eq:disjoint-fragment-bound}
t_{k_0}(J_L)\le L(k_0+\tau).
\end{equation}
\end{claim}

\begin{proof}
Let $\mathcal X$ be a pairwise disjoint family of $k_0$-deficient fragments
of $J_L$. For $X\in\mathcal X$ and $i\in\{1,\ldots,L\}$, put
$Y_i(X)=X\cap V_i$, viewed as a subset of the corresponding copy of $H$, and
let $I_X=\{i:Y_i(X)\ne\varnothing\}$. Since $X$ is nonempty, $I_X$ is
nonempty. There are no edges between distinct copies of $H$, so
\begin{equation}\label{eq:neighborhood-sum}
|N_{J_L}(X)|=\sum_{i\in I_X}|N_H(Y_i(X))|.
\end{equation}
Since $X$ is $k_0$-deficient, the nonnegative summands in
\eqref{eq:neighborhood-sum} have sum smaller than $k_0$; hence every summand
is smaller than $k_0$. Furthermore, the difference between the right- and
left-hand sides below is $(|I_X|-1)k_0\ge0$, and hence
\begin{equation}\label{eq:deficiency-split}
k_0-|N_{J_L}(X)|
\le\sum_{i\in I_X}\bigl(k_0-|N_H(Y_i(X))|\bigr).
\end{equation}

Fix $i\in\{1,\ldots,L\}$. The nonempty sets $Y_i(X)$, where
$X\in\mathcal X$, are pairwise disjoint. Those satisfying
$V_i\setminus(Y_i(X)\cup N_H(Y_i(X)))\ne\varnothing$ are
$k_0$-deficient fragments of $H$. They form a pairwise disjoint, and hence
half-disjoint, family, so the sum of their deficiencies is at most $\tau$.

Consider the remaining nonempty sets. If there are none, their total
contribution to the right-hand side of \eqref{eq:deficiency-split} is zero.
Otherwise, denote them by $Z_1,\ldots,Z_r$, where $r\ge1$. For each $j$, we
have $V(H)\setminus(Z_j\cup N_H(Z_j))=\varnothing$, so
$N_H(Z_j)=V(H)\setminus Z_j$; equivalently, every vertex outside $Z_j$ has a
neighbor in $Z_j$. The sets $Z_1,\ldots,Z_r$ are pairwise disjoint, and
$n\ge k_0+1$. It follows that
\begin{align*}
\sum_{j=1}^{r}\bigl(k_0-|N_H(Z_j)|\bigr)
=r(k_0-n)+\sum_{j=1}^{r}|Z_j|
\le r(k_0-n)+n
=k_0+(r-1)(k_0-n)\le k_0.
\end{align*}
Thus, for each fixed $i$, the sum of the terms on the right-hand side of
\eqref{eq:deficiency-split} is at most $k_0+\tau$. Summing first over
$X\in\mathcal X$ and then over $i\in\{1,\ldots,L\}$ gives the inequality in
Claim~\ref{clm:t-bound} for the family $\mathcal X$. Taking the maximum over
all pairwise disjoint families $\mathcal X$ proves the claim.
\end{proof}

Let $c$ be the number of connected components of $H$. We have
\begin{equation}\label{eq:component-bound}
c\le\frac{k_0+\tau}{2}.
\end{equation}
To verify this, first suppose that $c=1$. Since $k_0\ge2$ and $\tau\ge0$,
the right-hand side of \eqref{eq:component-bound} is at least one. If
$c\ge2$, the vertex sets of the components of $H$ form a pairwise disjoint,
and hence half-disjoint, family of $k_0$-deficient fragments. Each member has
neighborhood of size zero, so $\tau\ge ck_0\ge2c$, which again gives
\eqref{eq:component-bound}.

\begin{claim}\label{clm:shredder-bound}
For every integer $L\ge2$,
\begin{equation}\label{eq:shredder-bound}
\widehat b_{k_0}(J_L)-1
\le\frac{L}{2}(k_0+\tau)+(k_0-1)n+k_0-1.
\end{equation}
\end{claim}

\begin{proof}
If $J_L$ has no $(k_0-1)$-shredder, then
$\widehat b_{k_0}(J_L)=0$, and the inequality holds. Otherwise, let $K$ be
an arbitrary $(k_0-1)$-shredder of $J_L$, and let $q$ be the number of
copies of $H$ that meet $K$. Since $|K|=k_0-1$, we have $q\le k_0-1$.
Each of the $L-q$ copies disjoint from $K$ contributes $c$ components to
$J_L-K$, whereas each of the other $q$ copies contributes at most $n$
components. Hence
\[
b_{J_L}(K)\le(L-q)c+qn\le Lc+(k_0-1)n.
\]
The definition of $\delta_{J_L}(K)$ gives $\delta_{J_L}(K)\le k_0$.
Combining these two inequalities with \eqref{eq:component-bound}, and then
taking the maximum over all $(k_0-1)$-shredders $K$, proves the claim.
\end{proof}

By Claim~\ref{clm:t-bound},
$\lceil t_{k_0}(J_L)/2\rceil\le L(k_0+\tau)/2+1$.
Let $B=\max\{1,(k_0-1)n+k_0-1\}$. This constant depends only on $k_0$ and
$n$, and is independent of $L$. For all sufficiently large $L$, inequality
\eqref{eq:a-lower} gives
$a_{k_0}(J_L)\ge C_{k_0}=10(k_0+2)^3(k_0+1)^3$.
Since $k_0\ge2$ and $L\ge2$, the graph $J_L$ is disconnected and has
$Ln\ge k_0+1$ vertices. All hypotheses of
Theorem~\ref{thm:JJ-connectivity} hold. Applying it together with Claim~\ref{clm:shredder-bound} and the preceding inequality, we obtain
\begin{equation}\label{eq:a-upper}
a_{k_0}(J_L)\le\frac{L}{2}(k_0+\tau)+B.
\end{equation}
Combining \eqref{eq:a-lower} and \eqref{eq:a-upper}, and using $k_0=2D$,
we have
\[
L(D+\Delta)-D
\le a_{k_0}(J_L)
\le LD+\frac{L}{2}\tau+B.
\]
After cancelling $LD$ and dividing by $L$, this gives
$\Delta\le\tau/2+(B+D)/L$. Letting $L$ tend to infinity yields
$t'_{k_0}(H)=\tau\ge2\Delta$, as required.
\end{proof}

The one-dimensional case has a direct graph-theoretic form. Recall that
$r_1(H)=n-c(H)$, where $c(H)$ is the number of connected components of $H$.

\begin{lemma}\label{lem:rank-bound-one}
Let $H$ be a graph on $n\ge3$ vertices and put
$\Delta_1(H)=n-1-r_1(H)$. Then
\[
t'_2(H)\ge2\Delta_1(H).
\]
\end{lemma}

\begin{proof}
If $H$ is connected, then $\Delta_1(H)=0$ and there is nothing to prove.
Suppose that $H$ has $c\ge2$ connected components. Their vertex sets form a
pairwise disjoint, and hence half-disjoint, family of $2$-deficient
fragments. Each has neighborhood of size zero. Therefore
\[
t'_2(H)\ge2c\ge2(c-1)=2\Delta_1(H).
\]
\end{proof}

\subsection{\bhz{Proofs of packing theorems and corollaries}{}}
\label{subsec:proof-rank-packing}

We now combine the preceding estimates for several dimensions.

\rankdefectthm*
Equivalently,
\[
\sum_{i=1}^{s}\bigl(d_i n-D_i-r_{d_i}(H)\bigr)\le |E_F|.
\]
\begin{proof}
For each $i\in\{1,\ldots,s\}$, put $k_i=d_i(d_i+1)$ and
$\Delta_i=d_i n-D_i-r_{d_i}(H)$.
Since $H\cup F$ is $K$-connected, we have $n\ge K+1$. By
Lemma~\ref{lem:fragment-edge-bound},
\begin{equation}\label{eq:mixed-fragment-upper}
t'_K(H)\le2|E_F|.
\end{equation}
Fix $i$. Lemma~\ref{lem:t-prime-scaling}, with $a=k_i$ and $b=K$, gives
\[
t'_K(H)\ge\frac{K}{k_i}t'_{k_i}(H).
\]
If $d_i\ge2$, Lemma~\ref{lem:rank-bound} gives
$t'_{k_i}(H)\ge2\Delta_i$; if $d_i=1$, the same inequality follows from
Lemma~\ref{lem:rank-bound-one}. Hence \eqref{eq:mixed-fragment-upper} yields
\[
\Delta_i\le\frac{k_i}{K}|E_F|.
\]
Summing over $i$ and using $\sum_i k_i=K$, we obtain
\[
\sum_{i=1}^{s}\Delta_i\le |E_F|.
\]
This is equivalent to \eqref{eq:rank-defect}.
\end{proof}

The inequality in Theorem~\ref{thm:rank-defect} is also necessary for any such
packing. Indeed, if an $n$-vertex graph $Q=(V,E)$ with $n \ge \max_{i} d_i$ contains pairwise edge-disjoint
$d_i$-rigid spanning subgraphs, choose pairwise disjoint bases $B_i$ of
$\Rmat_{d_i}(K(V))$ and put $\rho_i=d_i n-D_i$. For every partition
$E=E_H\dotunion E_F$, with $H=(V,E_H)$,
\[
\sum_{i=1}^{s}\rho_i
=\sum_{i=1}^{s}|B_i\cap E_H|+\sum_{i=1}^{s}|B_i\cap E_F|
\le\sum_{i=1}^{s}r_{d_i}(H)+|E_F|.
\]
Thus Theorem~\ref{thm:rank-defect} provides precisely the matroid-union
inequality needed under the connectivity hypothesis.

\mainthm*
\begin{proof}
Let $G=(V,E)$ be a $K$-connected graph, let $n=|V|$, and put
$D_i=\binom{d_i+1}{2}$ and $\rho_i=d_i n-D_i$. Since $n\ge K+1>d_i$ for
every $i$, $\rho_i$ is the full rank of the $d_i$-dimensional rigidity
matroid on $V$.

Apply \jbj{Theorem}~\ref{thm:edmonds} to the matroids
$\Rmat_{d_1}(G),\ldots,\Rmat_{d_s}(G)$. Fix $X\subseteq E$, and put
$H=(V,X)$ and $F=(V,E\setminus X)$. Since $H\cup F=G$,
Theorem~\ref{thm:rank-defect} gives
\begin{equation}\label{eq:union-condition}
|E\setminus X|+\sum_{i=1}^{s}r_{d_i}(X)
\ge\sum_{i=1}^{s}\rho_i.
\end{equation}
As this holds for every $X\subseteq E$, \jbj{Theorem}~\ref{thm:edmonds}
implies that at least $\sum_i\rho_i$ edges of $G$ can be selected and
partitioned into pairwise disjoint sets $I_1,\ldots,I_s$, where $I_i$ is
independent in $\Rmat_{d_i}(G)$. Since $|I_i|\le\rho_i$ for every $i$, no
such selection can contain more than $\sum_i\rho_i$ edges. Hence the equality
holds, and therefore $|I_i|=\rho_i$ for every $i$.

Thus each $I_i$ is a base of $\Rmat_{d_i}(K(V))$, and the spanning graph
$(V,I_i)$ is $d_i$-rigid. The sets $I_1,\ldots,I_s$ are pairwise disjoint,
which proves the sufficient condition.

If $K\ge4$, Lemma~\ref{lem:lower-bound-construction}, applied directly to
$d_1,\ldots,d_s$, gives infinitely many $(K-1)$-connected graphs with no
such packing. Hence the bound is sharp in every nontrivial case. If $K=2$,
then $s=1$ and $d_1=1$, and the assertion itself is immediate from ordinary
connectivity.
\end{proof}

\begin{proof}[Proof of Corollary~\ref{cor:equal-packings}]
For part~\textup{(a)}, apply Theorem~\ref{thm:main} with $s=t$ and
$d_1=\cdots=d_t=d$. Sharpness follows from
Lemma~\ref{lem:lower-bound-construction}.

For part~\textup{(b)}, suppose first that $k\ge2$ and apply
part~\textup{(a)} with $d=k$. The graph has at least
$tk(k+1)+1\ge k+1$ vertices, so every resulting $k$-rigid spanning
subgraph is $k$-connected by Lemma~\ref{lem:standard}. If $k=1$, then the
graph is $2t$-edge-connected, and the spanning-tree packing theorem of
Nash-Williams and Tutte \cite{NashWilliams,Tutte} gives $t$ pairwise
edge-disjoint spanning trees.
\end{proof}

\begin{proof}[Proof of Corollary~\ref{cor:tree-removal}]
For part~\textup{(a)}, apply Theorem~\ref{thm:main} with the dimension
sequence $d,1,\ldots,1$, where $1$ occurs $r$ times. This gives a
$d$-rigid spanning subgraph $H$ and pairwise edge-disjoint connected
spanning subgraphs $Q_1,\ldots,Q_r$, all of which are edge-disjoint from
$H$. Choose a spanning tree $T_i$ of each $Q_i$. Sharpness follows from
Lemma~\ref{lem:lower-bound-construction} applied to the same dimension
sequence.

For part~\textup{(b)}, suppose first that $k\ge2$ and apply
part~\textup{(a)} with $d=k$. After the edges of
$T_1,\ldots,T_r$ are deleted, the remaining graph contains the
$k$-rigid spanning subgraph $H$, which is $k$-connected by
Lemma~\ref{lem:standard}. If $k=1$, then $G$ is
$2(r+1)$-edge-connected, so Theorem~\ref{thm:tree-packing} gives
$r+1$ pairwise edge-disjoint spanning trees. Deleting any $r$ of them
leaves the remaining spanning tree.
\end{proof}

\section{A general bound for highly connected orientations}\label{sec:explicit-orientation}

For a digraph, a {\bf strong component} is a maximal vertex set inducing a
strongly connected subdigraph. Its {\bf condensation} is obtained by
contracting the strong components; a source component has in-degree zero and
a sink component has out-degree zero in the condensation. A vertex set is \jbj{{\bf independent}} if no two of its vertices are adjacent.

The proof uses two edge-disjoint minimally $(2q-2)$-rigid spanning subgraphs
and a common independent set $C$ of size $q+1$. Prescribed in-degrees and
out-degrees force every source and sink strong component remaining after the
deletion of fewer than $q$ vertices to meet $C$. The unused edges then
provide directed paths that place the surviving vertices of $C$ in a single
strong component.

\subsection{An independent-set estimate}

As usual, $\alpha(F)$ denotes the maximum size of an independent set in $F$. \jbj{The following consequence of Tur\'an's theorem will provide the common
set $C$.}

\begin{lemma}\label{lem:independent}
Let $q\ge2$, put $d=2q-2$ and $D=\binom{d+1}{2}$, and let $F$ be an
$n$-vertex graph with $n\ge8q(q-1)+1$. If $|E(F)|\le 2dn-2D$, then
$\alpha(F)\ge q+1$.
\end{lemma}

\begin{proof}
Put $m=|E(F)|$ and $n_0=8q(q-1)+1$. The standard bound obtained from
Tur\'an's theorem gives
\[
\alpha(F)\ge \frac{n^2}{2m+n}
\ge \frac{n^2}{(8q-7)n-4(q-1)(2q-1)}.
\]
To compare the last expression with $q$, define
$\phi(x)=x^2-q\bigl((8q-7)x-4(q-1)(2q-1)\bigr)$.
A direct calculation gives
$\phi(n_0)=(q-1)(4q-1)>0$. Moreover, for every integer $x\ge n_0$,
\[
\phi(x+1)-\phi(x)=2x+1-q(8q-7)
\ge 8q^2-9q+3>0.
\]
Hence $\phi(n)>0$, and the displayed lower bound for $\alpha(F)$ is greater
than $q$. Since $\alpha(F)$ is an integer, $\alpha(F)\ge q+1$.
\end{proof}

\subsection{Orientations with prescribed degrees}
The independence of $C$ allows the degrees prescribed on $C$ to be smaller
than those prescribed at the other vertices. For an orientation $D$ and a
vertex $v$, write $d_D^-(v)$ and $d_D^+(v)$ for the in-degree and out-degree
of $v$, respectively.

\begin{lemma}\label{lem:independent-prescribed-degrees}
Let $q\ge3$, put $d=2q-2$, and let $J$ be a minimally $d$-rigid graph on at least $d+1$ vertices. Suppose that
$C=\{c_0,c_1,\ldots,c_q\}$ is independent in $J$. Define
$g:V(J)\to\mathbb Z_{\ge0}$ by setting $g(c_0)=0$,
$g(c_1)=g(c_2)=g(c_3)=2$, $g(c_i)=3$ for $4\le i\le q$, and
$g(v)=d$ for $v\notin C$.
Then $J$ has \jbj{orientations $O_1$ and $O_2$ with the following properties}.
\begin{enumerate}[label=\textup{(\roman*)}]
\item For every $v\in V(J)$, $d^-_{O_1}(v)=g(v)$ and
$d^+_{O_2}(v)=g(v)$.
\item Let $S\subseteq V(J)$ with $|S|\le q-1$. Every nonempty set
$X\subseteq V(J)\setminus S$ that receives no arc in $O_1$ from
$V(J)\setminus(X\cup S)$ satisfies $X\cap C\ne\varnothing$. Similarly,
every nonempty set $Y\subseteq V(J)\setminus S$ that sends no arc in $O_2$
to $V(J)\setminus(Y\cup S)$ satisfies $Y\cap C\ne\varnothing$.
\end{enumerate}
\end{lemma}

\begin{proof}
Put $n=|V(J)|$ and $D=\binom{d+1}{2}=(q-1)(2q-1)$. Since $J$ is minimally
$d$-rigid, its edge set is independent in the $d$-dimensional rigidity
matroid and $|E(J)|=dn-D$. The definition of $g$ gives
$g(C)=3q-3=d(q+1)-D$, and hence $g(V(J))=dn-D=|E(J)|$.
Moreover, $0\le g(c)\le d$ for every $c\in C$, and
$\sum_{c\in C}(d-g(c))=D$.

We verify the subset inequalities in Theorem~\ref{thm:hakimi}. Let
$X\subseteq V(J)$. If $|X|\ge d$, rigidity-matroid independence gives
\[
\jbj{|E(J[X])|\le d|X|-D
=d|X|-\sum_{c\in C}(d-g(c))
\le d|X|-\sum_{c\in X\cap C}(d-g(c))=g(X).}
\]

If $|X|\le d-1$, put $a=|X\cap C|$ and $b=|X\setminus C|$.
Since $C$ is independent and $a+b\le d-1$, we have
\[
|E(J[X])|\le ab+\binom b2
=b\left(a+\frac{b-1}{2}\right)\le db\le g(X).
\]
Thus Theorem~\ref{thm:hakimi} gives an orientation $O_1$ with the prescribed
in-degrees. Let $O_2$ be the orientation obtained by reversing every arc of
$O_1$; then $d^+_{O_2}(v)=g(v)$ for every $v\in V(J)$.

\jbj{Let $S\subseteq V(J)$ satisfy $|S|\le q-1$, and let
$X\subseteq V(J)\setminus S$ be a nonempty set that receives no arc in
$O_1$ from $V(J)\setminus(X\cup S)$. Suppose that $X\cap C=\varnothing$,
and write $x=|X|$ and $s=|S|$. Every vertex of $X$ has in-degree $d$, and
every arc entering $X$ uses an edge of $J[X\cup S]$. If $x+s\ge d$, then}
\[
\jbj{dx\le |E(J[X\cup S])|\le d(x+s)-D,}
\]
\jbj{so $ds\ge D$. This is impossible because
$ds\le d(q-1)=2(q-1)^2<D=(q-1)(2q-1)$. If $x+s\le d-1$, then every vertex
of $X$ has at most $x+s-1\le d-2$ possible in-neighbors, again contradicting
its in-degree $d$. Therefore $X\cap C\ne\varnothing$.}

\jbj{If a nonempty set $Y\subseteq V(J)\setminus S$ sends no arc in $O_2$ to
$V(J)\setminus(Y\cup S)$, then no arc of $O_1$ enters $Y$ from this set,
because $O_2$ is obtained by reversing every arc of $O_1$. The preceding
argument therefore gives $Y\cap C\ne\varnothing$.}
\end{proof}

\subsection{\texorpdfstring{Proof of Theorem~\ref{thm:orientation-explicit}}{Proof of the explicit orientation bound}}\label{subsec:explicit-orientation-proof}

\explicitorientationthm*
\begin{proof}
Let $q\ge3$, put $K=(25q^2+41q-16)/2$, and let $G$ be a $K$-connected
graph on $n$ vertices. Set $d=2q-2$, $D=\binom{d+1}{2}$, and
$\ell=q(q+1)$. Then $K=(25/2)\ell+4d$, which is an integer because
$\ell$ is even. Since $K-2d(d+1)=(9q^2+65q-24)/2>0$, the graph $G$ is
$2d(d+1)$-connected. Theorem~\ref{thm:main}, applied with $d_1=d_2=d$,
therefore gives two edge-disjoint $d$-rigid spanning subgraphs.
Delete edges from each until it is minimally $d$-rigid, and denote the
resulting graphs by $J_1$ and $J_2$.

Let $H$ be the spanning subgraph of $G$ with edge set
$E(G)\setminus(E(J_1)\cup E(J_2))$. Since $\delta(G)\ge K$ and
each $J_i$ has $dn-D$ edges,
\[
|E(H)|\ge\frac{Kn}{2}-2(dn-D)
=\frac{25}{4}\ell n+2D
\ge\frac{25}{4}\ell n-\frac{25}{2}\ell^2.
\]
Also $n\ge K+1>5\ell/2$. By Theorem~\ref{thm:linked-subgraph}, the graph
$H$ contains a $2\ell$-connected, $\ell$-linked subgraph $L$ satisfying
$|E(L)|\ge5\ell|V(L)|$.

Put $U=V(L)$ and $N=|U|$. Since $L$ is simple,
$5\ell N\le\binom N2$, so $N\ge10\ell+1$. Consider the graph
$F=J_1[U]\cup J_2[U]$. Each $E(J_i[U])$ is independent in
\jbj{$\Rmat_d(K(U))$, so $|E(F)|\le2dN-2D$. Since
$N\ge10q(q+1)+1>8q(q-1)+1$, Lemma~\ref{lem:independent} gives an independent
set $C=\{c_0,c_1,\ldots,c_q\}$ in $F$. In particular, $C$ is independent in
both $J_1$ and $J_2$.}

We next construct a path from $c_i$ to $c_j$ for every ordered pair
$(i,j)$ with $i\ne j$. These paths will have pairwise disjoint sets of
internal vertices and will meet $C$ only at their ends. For each such pair,
choose a vertex $b_{ij}\in N_L(c_j)\setminus C$, with all the $b_{ij}$
distinct. A greedy choice is possible because $\delta(L)\ge2\ell$, each
$c_j$ has at least $2\ell-q\ge\ell$ neighbors outside $C$, and there
are only $\ell$ choices to make.

By Theorem~\ref{thm:strongly-linked}, the graph $L$ is strongly
$\ell$-linked. Apply this property to the $\ell$ pairs
$\{c_i,b_{ij}\}$, one for each ordered pair $(i,j)$ with $i\ne j$.
We obtain paths $P_{ij}$ that may meet only at common prescribed ends.
Every vertex of $C$ and every $b_{ij}$ is a prescribed end, so no such
vertex is internal to any $P_{ij}$. Append the edge $b_{ij}c_j$ to
$P_{ij}$ and call the resulting path $Q_{ij}$.
The added edges are distinct and occur in none of the paths $P_{ab}$.
Indeed, an edge $b_{ij}c_j$ has two prescribed ends and could occur in
such a path only if that entire path joined these two ends, whereas the
unique prescribed pair containing $b_{ij}$ is $\{c_i,b_{ij}\}$ with
$i\ne j$.
It follows that the paths $Q_{ij}$ are edge-disjoint, their sets of
internal vertices are pairwise disjoint, and they meet $C$ only at their
ends. Orient each $Q_{ij}$ from $c_i$ to $c_j$.

Let $g$ be the function in Lemma~\ref{lem:independent-prescribed-degrees}
for this choice of $C$. Orient $J_1$ with in-degrees $g$ and $J_2$ with
out-degrees $g$. These orientations are compatible with those of the paths
$Q_{ij}$, since the paths lie in $H$. Orient every remaining edge of $G$
arbitrarily, and denote the resulting orientation by $\vec G$.

Let $S\subseteq V(G)$ have size at most $q-1$. We first show that all
vertices of $C\setminus S$ lie in one strong component of $\vec G-S$.
For distinct $c_i,c_j\in C\setminus S$, the directed path $Q_{ij}$,
together with the concatenations of $Q_{ik}$ and $Q_{kj}$ for
$k\notin\{i,j\}$, gives $q$ internally vertex-disjoint directed paths
from $c_i$ to $c_j$. At most $q-1$ of them can meet $S$, so at least one
survives. This holds for every ordered pair in $C\setminus S$, proving
the assertion.

Suppose that $\vec G-S$ is not strongly connected. Since $|S|<K$, the
underlying graph $G-S$ is connected. Hence the condensation of $\vec G-S$
is an acyclic digraph with at least two vertices and has a connected
underlying graph. It has distinct source and sink components; denote their vertex sets
by $X$ and $Y$, respectively. No arc enters $X$ from the other
vertices of $\vec G-S$; in particular, none does so in the orientation
of $J_1-S$. Lemma~\ref{lem:independent-prescribed-degrees} gives
$X\cap C\ne\varnothing$. The out-degree version applied to $J_2-S$
similarly gives $Y\cap C\ne\varnothing$. This contradicts the fact that
all vertices of $C\setminus S$ lie in one strong component.
Thus $\vec G-S$ is strongly connected for every $S$ of size at most
$q-1$, and $\vec G$ is a $q$-connected orientation of $G$.
\end{proof}

\section{Improved bounds for highly connected orientations}\label{sec:orientation}

We prove Theorem~\ref{thm:orientation} by combining two edge-disjoint
minimally $d$-rigid spanning subgraphs with a linkage in the remaining
edges. Prescribed in-degrees and out-degrees ensure that, after fewer than
$q$ vertex deletions, every source and sink strong component contains many
vertices of a common set $C$. We then realize the arcs of an auxiliary
oriented graph on $C$ by internally disjoint paths. Its arc counts across
the relevant partitions yield the final contradiction.

\subsection{Orientations of rigid subgraphs}\label{subsec:rigid-orientations}

We first show that the required in-degrees can be prescribed in a minimally
rigid spanning subgraph. Reversing all arcs gives the corresponding
orientation with prescribed out-degrees.

\begin{lemma}\label{lem:prescribed-degrees}
Let $d\ge2$ be even, let $J$ be a minimally $d$-rigid graph on at least
$d+1$ vertices, and let $C\subseteq V(J)$ have size $d+1$. Define
$g:V(J)\to\mathbb Z_{\ge0}$ by
\[
g(v)=\begin{cases}
d/2,&v\in C,\\
d,&v\notin C.
\end{cases}
\]
Then $J$ has an orientation in which every vertex $v$ has in-degree $g(v)$.
\end{lemma}

\begin{proof}
Put $n=|V(J)|$ and $D=\binom{d+1}{2}$. Since $J$ is minimally $d$-rigid,
its edge set is independent in the $d$-dimensional rigidity matroid and
$|E(J)|=dn-D$. Moreover,
\[
g(V(J))=d(n-d-1)+\frac d2(d+1)=dn-D=|E(J)|.
\]
It remains to verify the subset inequalities in Theorem~\ref{thm:hakimi}.
Let $X\subseteq V(J)$. If $|X|\ge d$, then rigidity-matroid independence
gives $|E(J[X])|\le d|X|-D$, while
\[
g(X)=d|X|-\frac d2|X\cap C|\ge d|X|-D.
\]
If $|X|\le d-1$, then $J$ is simple and
\[
|E(J[X])|\le\binom{|X|}{2}\le\frac d2|X|\le g(X).
\]
Thus Hakimi's criterion applies.
\end{proof}

With $d=2q+26$, the prescribed in-degrees ensure that, after fewer than
$q$ vertex deletions, every nonempty vertex set receiving no arc from the
other remaining vertices contains at least $29$ vertices of $C$. Reversing
all arcs gives the corresponding out-degree statement for sets with no
outgoing arc.

\begin{lemma}\label{lem:vertices-in-C}
Let $q\ge2$, put $d=2q+26$, and let $J$ be a minimally $d$-rigid graph on
at least $d+1$ vertices.
Let $C\subseteq V(J)$ have size $d+1$, and orient $J$ as in
Lemma~\ref{lem:prescribed-degrees}. If $S\subseteq V(J)$ satisfies
$|S|\le q-1$ and a nonempty set $X\subseteq V(J)\setminus S$ receives no arc
from $V(J)\setminus(X\cup S)$, then $|X\cap C|\ge29$.

In the orientation obtained by reversing every arc, the analogous statement
holds for a set with no outgoing arc; its out-degrees are prescribed by the
same function $g$.
\end{lemma}

\begin{proof}
Write $x=|X|$, $s=|S|$, $a=|X\cap C|$, and
$D=\binom{d+1}{2}$. Since no arc enters $X$ from outside $X\cup S$, all
in-arcs counted at vertices of $X$ use edges of $J[X\cup S]$. Hence
\begin{equation}\label{eq:incoming-arc-count}
dx-\frac d2a\le |E(J[X\cup S])|.
\end{equation}
Suppose first that $x+s\ge d$. Rigidity-matroid independence and
\eqref{eq:incoming-arc-count} give
\[
dx-\frac d2a\le d(x+s)-D,
\]
and therefore
\[
a\ge d+1-2s\ge (2q+27)-2(q-1)=29.
\]

Suppose next that $x+s\le d-1$. If $X$ contained a vertex of
$V(J)\setminus C$, then that vertex would have prescribed in-degree $d$, but
all of its in-neighbors would lie in $(X\cup S)\setminus\{v\}$, a set of
size at most $d-2$, which is impossible. Thus $X\subseteq C$. Every vertex
of $X$ has in-degree $d/2$, and all these in-arcs have their tails in
$X\cup S$. Since $J$ is simple,
\[
\frac d2x\le\binom{x}{2}+sx.
\]
As $x>0$, this yields $x\ge d+1-2s\ge29$, as required. The out-degree
version follows by reversing every arc.
\end{proof}

\subsection{A probabilistic construction}\label{subsec:auxiliary-orientation}

The next lemma constructs an auxiliary oriented graph with prescribed bounds
on the total number of arcs and on the number of arcs in each direction
between the parts of each specified partition. In the proof of
Theorem~\ref{thm:orientation}, these arcs will be realized by paths in the
\jbj{spanning subgraph obtained by deleting the edges of $J_1$ and $J_2$}.

For an oriented graph $D$, let $A(D)$ denote its arc set. For a nonnegative
integer $N$ and $p\in[0,1]$, we write $\operatorname{Bin}(N,p)$ for the
binomial distribution with parameters $N$ and $p$.

\begin{lemma}\label{lem:auxiliary-oriented-graph}
For all sufficiently large integers $q$, put $d=2q+26$ and
$r=\left\lfloor\frac{4}{25}d(d-1)\right\rfloor$.
Let $C$ be a set of size $d+1$. There is an oriented graph $\Gamma$ on $C$
with at most $r$ arcs such that, for every partition
$C=A\dotunion B\dotunion Z$ satisfying
$|A|,|B|\ge29$ and $|Z|\le q-1$, there are at least $q$ arcs of $\Gamma$
from $A$ to $B$ and at least $q$ arcs from $B$ to $A$.
\end{lemma}


\begin{proof}
All uses of $O(\cdot)$, $\Theta(\cdot)$, and $o(1)$ in this proof refer to
$q\to\infty$, equivalently $d=2q+26\to\infty$, with absolute implicit
constants. Construct a random oriented graph $\Gamma$ on $C$ as follows.
For each unordered pair $\{u,v\}\subseteq C$, independently choose the arc
$uv$ with probability $5/32$, choose the arc $vu$ with probability $5/32$,
and choose neither arc with probability $11/16$.

Fix a partition $C=A\dotunion B\dotunion Z$ as in the statement. Since
$|C|=2q+27$ and $|Z|\le q-1$, we have $|A|+|B|\ge q+28$. For sufficiently
large $q$, the product is minimized when the smaller part has size $29$;
hence $|A||B|\ge29(q-1)$.
Let $N=29(q-1)$ and let $Y\sim\operatorname{Bin}(N,5/32)$. The number of
arcs from $B$ to $A$ has distribution
$\operatorname{Bin}(|A||B|,5/32)$. Since $|A||B|\ge N$, the probability
that fewer than $q$ arcs go from $B$ to $A$ is at most $\Pr(Y<q)$. For
$0\le j\le q-1$, consecutive terms of the binomial distribution satisfy
\[
\frac{\Pr(Y=j+1)}{\Pr(Y=j)}
=\frac{N-j}{j+1}\cdot \frac5{27}
\ge \frac{140(q-1)}{27q}>1.
\]
Consequently, the summands in $\Pr(Y<q)$ are increasing. Using
$\binom Nq\le(eN/q)^q\le(29e)^q$ and $N-q=28q-29$, we obtain
\begin{align*}
\Pr(Y<q)
\le q\binom Nq\left(\frac5{32}\right)^q
       \left(\frac{27}{32}\right)^{N-q}
\le q\left(\frac{27}{32}\right)^{-29}
\left(29e\frac5{32}\left(\frac{27}{32}\right)^{28}\right)^q.
\end{align*}
There are at most $3^{|C|}=3^{2q+27}$ ordered partitions of $C$ into
$A,B,Z$. Applying the same estimate to both directions and taking a union
bound, the probability that some admissible partition fails the required
condition is at most
\[
2\cdot3^{27}q\left(\frac{27}{32}\right)^{-29}
\left(9\cdot29e\frac5{32}
\left(\frac{27}{32}\right)^{28}\right)^q=o(1),
\]
because
$9\cdot29e\frac5{32}\left(\frac{27}{32}\right)^{28}<0.953<1$.

Let $M=|A(\Gamma)|$. An unordered pair receives an arc with probability
$5/16$. Therefore $\mathbb E M=5d(d+1)/32$ and
$\operatorname{Var}(M)=O(d^2)$. Since
$r\ge4d(d-1)/25-1$, we have
\[
r-\mathbb E M
\ge\frac4{25}d(d-1)-1-\frac5{32}d(d+1)
=\frac3{800}d^2-\frac{253}{800}d-1=\Theta(d^2).
\]
Chebyshev's inequality therefore gives
$\Pr(M>r)\le \operatorname{Var}(M)/(r-\mathbb E M)^2
=O(d^{-2})=o(1)$. Hence, with positive probability, the
required arc counts hold for every admissible partition and $M\le r$. Such an
oriented graph $\Gamma$ has the required properties.
\end{proof}

\subsection{\texorpdfstring{Proof of Theorem~\ref{thm:orientation}}{Proof of the orientation theorem}}\label{subsec:orientation-proof}

We now combine the preceding lemmas with Theorem~\ref{thm:main}, applied with
$d_1=d_2=d$, and Theorems~\ref{thm:linked-subgraph}
and~\ref{thm:strongly-linked}.

\orientationthm*
\begin{proof}
Let $q$ be sufficiently large, put $d=2q+26$ and
$D=\binom{d+1}{2}$, and set
$r=\left\lfloor4d(d-1)/25\right\rfloor$ and
$K=2d(d+1)=8q^2+212q+1404$.

Let $G$ be a $K$-connected graph on $n$ vertices. By
Theorem~\ref{thm:main} with $d_1=d_2=d$, $G$ contains two edge-disjoint $d$-rigid spanning
subgraphs. Delete redundant edges from them to obtain edge-disjoint minimally
$d$-rigid spanning subgraphs $J_1$ and $J_2$. Thus
$|E(J_i)|=dn-D$ for $i\in\{1,2\}$.

Let $H=G-(E(J_1)\cup E(J_2))$. Since $\delta(G)\ge K$,
$|E(G)|\ge Kn/2=d(d+1)n$, and hence
\begin{equation}\label{eq:residual-density}
|E(H)|\ge d(d-1)n+d(d+1).
\end{equation}
Moreover, $n\ge K+1$, and
\[
\frac52r\le\frac25d(d-1)<2d(d+1)<n.
\]
Here the first inequality follows from $r\le4d(d-1)/25$, the middle
inequality is immediate for \jbj{$d>0$}, and the last follows from
$n\ge2d(d+1)+1$. The same bound on $r$, together with
\eqref{eq:residual-density}, gives
\[
|E(H)|\ge d(d-1)n+d(d+1)
\ge\frac{25}{4}rn+d(d+1)
>\frac{25}{4}rn-\frac{25}{2}r^2.
\]

Theorem~\ref{thm:linked-subgraph} therefore gives a $2r$-connected subgraph
$L\subseteq H$ that is $r$-linked. Since
$2r\ge\frac8{25}d(d-1)-2\ge d$ for sufficiently large $d$, we have
$|V(L)|\ge2r+1\ge d+1$. Fix a set
$C\subseteq V(L)$ of size $d+1$; \jbj{it need not be independent in $L$}.

Apply Lemma~\ref{lem:auxiliary-oriented-graph} to obtain an oriented graph
$\Gamma$ on $C$ with $m=|A(\Gamma)|\le r$ and
\jbj{the stated partition property}. Since
$m\le r$ and $|V(L)|\ge2r+1$, the graph $L$ is also $m$-linked. Indeed, any
$m$-linkage problem with distinct prescribed ends can be extended to an
$r$-linkage problem by choosing $2(r-m)$ unused vertices and pairing them;
discarding the auxiliary paths from an $r$-linkage gives the required
$m$-linkage. Moreover, $|V(L)|\ge2r+1\ge2m$. Theorem~\ref{thm:strongly-linked}
implies that $L$ is strongly $m$-linked. We may therefore choose, for every
arc $uv$ of
$\Gamma$, a $u$--$v$ path $P_{uv}$ in $L$ so that these paths are pairwise
internally disjoint and meet only at common prescribed ends. Orient each
$P_{uv}$ from $u$ to $v$, \jbj{and denote the resulting directed path by
$\vec P_{uv}$}. Since $\Gamma$ is an oriented simple graph and the paths of
the linkage meet only at prescribed common ends, no undirected edge lies on
two of these paths. In particular, no edge is required \jbj{to be oriented} in both
directions, so the path orientations are compatible.

Let $g$ be the function in Lemma~\ref{lem:prescribed-degrees}. Orient
$J_1$ \jbj{as $\vec J_1$} with in-degrees $g$. Orient $J_2$
\jbj{as $\vec J_2$} by
first taking an orientation with in-degrees $g$ and then reversing every arc,
so its out-degrees are given by $g$. The paths $P_{uv}$ lie in $L\subseteq H$
and hence use no edge of $J_1\cup J_2$. \jbj{Thus the orientations of
$\vec J_1$, $\vec J_2$, and the paths $\vec P_{uv}$ are pairwise compatible.}
Orient every
remaining edge of $G$ arbitrarily, and denote the resulting orientation by
$\vec G$.

We \jbj{now} prove that $\vec G$ is $q$-connected. Let $S\subseteq V(G)$ with
$|S|\le q-1$ and suppose that $\vec G-S$ is not strongly connected. Since
$|S|<K$, the underlying graph $G-S$ is connected. The condensation of
$\vec G-S$ is therefore an acyclic digraph with at least two vertices and
has a connected underlying graph. \jbj{Choose a source component $X$ and a
sink component $Y$. They are distinct, and hence
$X\cap Y=\varnothing$.}

No arc of $\vec J_1$ enters $X$ from outside $X\cup S$.
Lemma~\ref{lem:vertices-in-C} therefore gives
$|X\cap C|\ge29$. Similarly, no arc of $\vec J_2$ leaves
$Y$ for a vertex outside $Y\cup S$, and the out-degree version of
Lemma~\ref{lem:vertices-in-C} gives $|Y\cap C|\ge29$.
Set $A=X\cap C$, $Z=S\cap C$, and $B=C\setminus(A\cup Z)$.
Then $|A|\ge29$, $|Z|\le q-1$, and $Y\cap C\subseteq B$, so $|B|\ge29$.
By the defining property of \jbj{the oriented graph $\Gamma$}, there are at
least $q$ arcs from $B$
to $A$. The corresponding directed paths $\vec P_{uv}$ have endpoints in
$C\setminus S$ and are pairwise internally disjoint. Since $|S|\le q-1$,
at least one of these $q$ directed paths avoids $S$ entirely. It starts at a
vertex of $B\subseteq C\setminus X$ and ends at a vertex of $A\subseteq X$.
The first arc of this path entering $X$ is therefore an arc of
$\vec G-S$ entering the source component $X$, a contradiction.

Thus $\vec G-S$ is strongly connected for every $S$ with $|S|\le q-1$.
Hence $\vec G$ is $q$-connected. Finally,
$K=2(2q+26)(2q+27)=8q^2+212q+1404$, which proves the theorem.
\end{proof}

\section{Concluding remarks}\label{sec:remarks}

When all $d_i=1$, Theorem~\ref{thm:main} reduces to the classical
spanning-tree packing setting.

The sharpness of Corollary~\ref{cor:equal-packings}\textup{(a)} does not
imply that the bound on $f^{*}(k)$ in part~\textup{(b)} is sharp. Indeed, the
construction in Lemma~\ref{lem:lower-bound-construction} excludes a packing
of $k$-rigid spanning subgraphs but not necessarily a packing of
$k$-connected spanning subgraphs.

For positive integers $k$ and $r$, let $h_r(k)$ be the least integer such
that every $h_r(k)$-connected graph contains $r$ pairwise
edge-disjoint spanning trees whose simultaneous deletion leaves a
$k$-connected graph. Corollary~\ref{cor:tree-removal}\textup{(b)} gives
\[
h_r(k)\le k(k+1)+2r.
\]
The sharpness construction for Corollary~\ref{cor:tree-removal}\textup{(a)}
does not give the
same lower bound for $h_r(k)$, since a $k$-connected graph need not contain
a $k$-rigid spanning subgraph. A simple lower bound follows from the complete
graph $K_{k+2r}$. If $r$ edge-disjoint spanning trees are deleted from this
graph, the remaining number of edges is
\[
\binom{k+2r}{2}-r(k+2r-1)
=\frac{k(k+2r-1)}{2}
<\frac{k(k+2r)}{2},
\]
so the remainder cannot be $k$-connected. Hence
\[
k+2r\le h_r(k)\le k(k+1)+2r.
\]

\begin{problem}
Determine the exact value of $h_r(k)$ for each fixed $r$ and all
$k\ge2$.
\end{problem}

The algorithmic approach of Garamv\"olgyi, Jord\'an, Kir\'aly and Vill\'anyi
\cite[pp.~14--15]{GJKV} extends to the mixed-dimensional setting of
Theorem~\ref{thm:main}.

\begin{proposition}\label{prop:algorithm}
Let $d_1,\ldots,d_s$ be positive integers, put
$K=\sum_{i=1}^{s}d_i(d_i+1)$, and let $G$ be a $K$-connected graph. There is
a randomized algorithm that always outputs pairwise edge-disjoint minimally
$d_i$-rigid spanning subgraphs $H_1,\ldots,H_s$ and has expected polynomial
running time in $|V(G)|$.
\end{proposition}

\begin{proof}
Put $n=|V(G)|$, $D_i=\binom{d_i+1}{2}$,
$\rho_i=d_i n-D_i$, and $\rho=\sum_{i=1}^{s}\rho_i$. Since
$n\ge K+1>d_i$ for every $i$, $\rho_i$ is the full rank of
$\Rmat_{d_i}(K(V(G)))$.

Let $\Omega=\{1,\ldots,2\rho\}$. For each $i$, choose independently a
realization $p_i:V(G)\to\Omega^{d_i}$ by selecting all coordinates
independently and uniformly from $\Omega$, and let $M_i$ be the row matroid
over $\mathbb{Q}$ of the integer matrix $R(G,p_i)$. Use the polynomial-time
matroid union algorithm of Edmonds \cite{Edmonds,EdmondsPartition}, with
exact Gaussian elimination for the independence tests in the matroids
$M_1,\ldots,M_s$, to find a largest set of edges that can be partitioned as
$I_1\dotunion\cdots\dotunion I_s$, where $I_i$ is independent in $M_i$.
If $\sum_i|I_i|=\rho$, output the spanning graphs $(V(G),I_i)$; otherwise,
repeat with new independent choices of the realizations $p_i$.

We first verify that every output is correct. If a set $I_i$ is independent
in $M_i$, some square minor of the rows of $R(G,p_i)$ indexed by $I_i$ is
nonzero. The determinant of the corresponding minor of the symbolic rigidity
matrix is therefore a nonzero polynomial. Hence the same rows are independent
at a generic realization, so $I_i$ is independent in
$\Rmat_{d_i}(G)$ and $|I_i|\le\rho_i$. If
$\sum_i|I_i|=\rho=\sum_i\rho_i$, equality must hold for every $i$.
Consequently, each $I_i$ is a base of $\Rmat_{d_i}(K(V(G)))$, and
$(V(G),I_i)$ is minimally $d_i$-rigid.

It remains to bound the probability that one trial succeeds. By
Theorem~\ref{thm:main}, fix pairwise edge-disjoint minimally $d_i$-rigid
spanning subgraphs with edge sets $B_i$, so $|B_i|=\rho_i$. For each $i$,
choose a nonzero $\rho_i\times\rho_i$ minor of the symbolic rigidity matrix
whose rows are indexed by $B_i$, and let $P_i$ be its determinant. Since the
entries of a rigidity matrix are linear in the coordinates, $P_i$ has total
degree at most $\rho_i$. The product
\[
P=\prod_{i=1}^{s}P_i
\]
is a nonzero polynomial of total degree at most $\rho$. By the polynomial
identity bound of Schwartz \cite{Schwartz},
\[
\Pr(P=0)\le \frac{\rho}{|\Omega|}=\frac12.
\]
Thus, with probability at least $1/2$, every $B_i$ is independent in the
sampled matroid $M_i$. On this event the matroid union algorithm finds a
partition of total size at least $\rho$; by the preceding paragraph it cannot
find one of larger total size. Hence the trial succeeds. The expected number
of trials is therefore at most two.

Finally, the dimension parameters do not have to be fixed. Since $G$ is
$K$-connected, $K\le n-1$, and
\[
\sum_{i=1}^{s} d_i \le \frac{K}{2} \le \frac{n-1}{2},
\quad\text{and}\quad
\rho \le n\sum_{i=1}^{s} d_i = O(n^2).
\]
Hence $s=O(n)$, every sampled coordinate has $O(\log n)$ bits, and the total
size of the matrices $R(G,p_i)$ is polynomial in $n$. Exact Gaussian
elimination gives polynomial-time independence tests, and Edmonds' matroid
union algorithm uses a polynomial number of such tests. Thus each trial, and
therefore the whole algorithm in expectation, runs in polynomial time.
\end{proof}

\jbj{\noindent{}\bf Declaration about the use of AI:} we used ChatGPT for some of the results in Section \ref{sec:packing} and the probabilistic proof in Section \ref{subsec:auxiliary-orientation}. In particular, ChatGPT suggested using
the theorem of Jackson and Jordán on the disjoint union $J_L$ to obtain a
rigidity-rank bound, and then applying Edmonds' matroid union theorem to
$G$. The authors take full responsibility for all results in the paper.

\end{spacing}

\end{document}